\documentclass[11pt]{amsart}

\usepackage[T1]{fontenc}
\usepackage[utf8]{inputenc}
\usepackage{lmodern}
\usepackage{amsmath,amssymb,amsthm,mathtools,mathrsfs, amsrefs}
\usepackage[margin=1.05in]{geometry}
\usepackage[shortlabels]{enumitem}
\usepackage{tikz-cd}
\usepackage[colorlinks=true,linkcolor=blue,citecolor=blue,urlcolor=blue]{hyperref}
\usepackage{microtype}

\newtheorem{theorem}{Theorem}[section]
\newtheorem{lemma}[theorem]{Lemma}
\newtheorem{proposition}[theorem]{Proposition}
\newtheorem{corollary}[theorem]{Corollary}
\theoremstyle{definition}

\newcommand{\C}{\mathbb C}
\newcommand{\N}{\mathbb N}
\newcommand{\cT}{\mathcal T}
\newcommand{\cH}{\mathcal H}
\newcommand{\cF}{\mathcal F}
\newcommand{\cD}{\mathcal D}

\newcommand{\ev}{\mathrm{ev}}
\newcommand{\cU}{\mathcal U}
\newcommand{\cS}{\mathcal S}
\newcommand{\cX}{\mathcal X}

\title{Strict comparison and selflessness}
\author{Leonel Robert}

\begin{document}
\begin{abstract}
It is shown that an infinite-dimensional, simple, unital, monotracial C*-algebra with strict comparison with respect to its trace is selfless.
\end{abstract}		

\maketitle

\section{Introduction}

A C*-probability space $(A,\rho)$, consisting of a unital C*-algebra and a GNS-faithful state, is called selfless if for some free ultrafilter $\cU$ 
and C*-probability space $(C,\kappa)$ with $C\neq \C$, there is a $*$-homomorphism
$$
(A,\rho)*(C,\kappa)\to (A^{\cU},\rho^{\cU})
$$
fixing $A$. This concept was introduced in \cite{selfless}, and was used by the authors of \cite{amrutametal} to resolve the long-standing problem of proving
the strict comparison property for the C*-algebra $C_r^*(\mathbb F_2)$. Their approach, which relied on the rapid decay property for groups, was further applied and explored in \cite{vigdorovich1, hayesetal,raumthielvilalta}. In \cite{ozawa}, Ozawa introduced an array of methods to establish selflessness that do not rely on the rapid decay property, substantially extending previous results. Since then,
selflessness has been actively investigated. For some recent contributions, see \cite{floresetal,gould,hayesetal2,gaojungeetal,xinmaetal,bell,ohshima,omland,vigdorovich,lebars} (this list is not exhaustive).

If $(A,\rho)$ is selfless and $\rho$ is a trace, then $A$ is infinite-dimensional, simple, monotracial, and has strict comparison with respect to $\rho$ \cite{selfless}. We show here that these properties characterize selflessness in the tracial case. The nontracial case has already been resolved: if $(A,\rho)$ is selfless and $\rho$ is nontracial, then $A$ is simple and purely infinite \cite{selfless, gould}. Conversely, for any state $\rho$ on a simple purely infinite C*-algebra $A$, the C*-probability space $(A,\rho)$ is selfless \cite{ozawa}.

\begin{theorem}\label{mainthm}
Let $A$ be an infinite-dimensional, simple, unital, monotracial C*-algebra with strict comparison with respect to its trace $\tau$. Then $(A,\tau)$ is selfless.
\end{theorem}
 It was shown in \cite{selfless} that the Jiang--Su algebra $\mathcal Z$ and the UHF C*-algebra $\mathcal Q$ are selfless. The proof
 relied on Ozawa's embedding of $C_r^*(\mathbb F_2)$ in $\mathcal Z^{\cU}$ in \cite{ozawa2}, which in turn relied on strong convergence theorems for random matrix ensembles. 
The above theorem directly shows that $\mathcal Z$ and $\mathcal Q$ are selfless, and this in turn implies that $C_r^*(\mathbb F_2)$ embeds in both $\mathcal Z^{\cU}$ and $\mathcal Q^{\cU}$. Importantly, the embedding of $C_r^*(\mathbb F_2)$ into $\mathcal Q^{\cU}$ is not used as an ingredient in the proof of Theorem~\ref{mainthm}. We thus obtain a non-probabilistic proof of the Haagerup-Thorbj{\o}rnsen result that $C_r^*(\mathbb F_2)$ is MF \cite{HT} and of Ozawa's result that it embeds in $\mathcal Z^{\cU}$ \cite{ozawa2}.

\begin{corollary}[Haagerup--Thorbj{\o}rnsen, Ozawa]
\label{coro:HTO}
For any free ultrafilter $\cU$ over $\N$, the reduced group C*-algebra $C_r^*(\mathbb F_2)$ embeds in both $\mathcal Z^{\cU}$ and $\mathcal Q^{\cU}$. 
\end{corollary}

\subsection{Structure of the proof} The proof is divided into three parts. Part I starts with an embedding
$$
\iota\colon \cS_{\tau}\longrightarrow (\pi_\tau(A)'')^{\,\cU}
$$
obtained from Popa's theorem \cite{Popa}. Here
$
(\cS_{\tau},\tau_\Omega):=(A,\tau)*(C([-2,2]),\sigma)$,
where $\sigma$ has semicircular distribution.
From this input, a $*$-homomorphism $\Theta$ from the relative cone $\cD(A,\cS_{\tau})$ into $A^{\cU}$ is obtained that fixes $A$ and satisfies $p\circ\Theta=\iota\circ \ev_1$; see \eqref{diagram}. 
This produces a lift  $x=\Theta(ts_\tau)$ of the semicircular element $s_\tau$.  In Part II, strict comparison is used in $A^{\cU}$ to obtain $r\in A^{\cU}$ such that 
$$
r^*br=\tau^{\cU}(b)e,\qquad b\in \Theta(\cD),
$$
where $e=\Theta(1-t^2)$. Part III ``repairs'' the lift $x$ of $s_\tau$ by setting 
$$
y=x+r+r^*
$$
and showing that the assignment $a\mapsto a$, $s_\tau\mapsto y$,  extends to a $*$-homomorphism
from $\cS_{\tau}$ to $A^{\cU}$ fixing $A$, and thus witnessing selflessness of $A$.

\subsection{Use of artificial intelligence}
The proof was developed through an interactive collaboration with ChatGPT Plus, using the models GPT-5.6 and GPT-6.0. ChatGPT supplied the crucial mathematical ideas, while I brainstormed with ChatGPT, provided feedback, checked and revised the arguments, and organized the proof in its present form. The proof was developed in two stages. An initial version assumed in addition that $A$ is exact and tracially $\mathcal Z$-stable (in the sense of \cite{HO}). These assumptions were used only in the first part of the argument. Further interaction with ChatGPT led to their removal, using Popa's theorem and the relative-cone lifting Proposition~\ref{prop:relative-cone-lift}. I wrote the manuscript, with proofreading assistance from ChatGPT, and take responsibility for the correctness of the arguments and the exposition.

\section{Preliminaries}
\subsection{The Toeplitz-Pimsner C*-algebra of a cp map and of a state}
For background on Hilbert C*-modules, C*-correspondences, and the KSGNS construction, we refer to \cite{brown-ozawa,lance}.
We briefly review these objects here and introduce their notation.

Let $A$ be a unital C*-algebra and  $\phi\colon A\to A$ a cp map.
The KSGNS construction applied to $\phi$ yields
a C*-correspondence $\cH_{\phi}=A\otimes_{\phi} A$
over $A$ obtained  by ``separation and completion'' 
from the algebraic tensor product $A\odot A$
endowed with the inner product
$$
\langle a_1\otimes a_2,a_1'\otimes a_2'\rangle:=
a_2^*\phi(a_1^*a_1')a_2'.
$$
See \cite{Pimsner}. The C*-correspondence $\cH_{\phi}$
is singly generated by the canonical generator
$\xi_{\phi}$ with class representative $1\otimes 1$.

Let $\cF(\cH_{\phi})$ denote the Fock space
of $\cH_{\phi}$. Let $\ell_{\phi}$ be the left creation operator on $\cF(\cH_{\phi})$ associated to
the generator $\xi_{\phi}\in \cH_{\phi}$. We denote by $\cT(\cH_{\phi})$ the 
Toeplitz-Pimsner C*-algebra of $\cH_{\phi}$. We note that $\cT(\cH_{\phi})=C^*(A,\ell_{\phi})$ and that
$$
\ell_{\phi}^*a\ell_{\phi}=\phi(a),\qquad a\in A.
$$
Finally, we denote by $E_{\Omega}\colon \cT(\cH_{\phi})\to A$
the vacuum conditional expectation.

The C*-algebra $\cT(\cH_{\phi})$ has the following
universal property (\cite[Proposition~1.3]{fowler-raeburn} and \cite[Definitions~2.1 and~3.1]{katsura}). Let $\theta\colon A\to B$
be a unital $*$-homomorphism  and let  $T\in B$ be
such that
$$
T^*\theta(a)T=\theta(\phi(a)),\qquad a\in A.
$$
(Thus, the assignments $\xi\mapsto T$ and $a\mapsto \theta(a)$ extend to a covariant representation of $\cH_{\phi}$ on $B$.)
Then there exists a unique $*$-homomorphism $\psi\colon \cT(\cH_{\phi})\to B$ extending $\theta$ and such that $\psi(\ell_{\phi})=T$.

The semicircular element with covariance $\phi$ obtained from $\ell_{\phi}$ is defined as 
$$
s_{\phi}=\ell_{\phi}+\ell_{\phi}^*\in \cT(\cH_\phi).
$$
We set 
$$
\cS(\cH_{\phi}):=C^*(A,s_{\phi})\subseteq \cT(\cH_{\phi}).
$$

In the arguments below we make special use of the above formalism in the case that $\phi$ is given  by a tracial state, i.e., $\phi(a)=\tau(a)1$ for $a\in A$, with $\tau$ a tracial state. In this case we write simply  $\cT_{\tau}$ and $\cS_{\tau}$ for the Toeplitz-Pimsner C*-algebra  and the C*-algebra generated by $A$ and the semicircular element $s_{\tau}$, respectively. We note that $\cT_{\tau}$
comes endowed with a vacuum state given by
$$
\tau_{\Omega}=\tau\circ E_{\Omega},
$$
which restricts to a trace on $\cS_{\tau}$.

We have the following state-preserving isomorphisms:
\begin{align*}
(\cT_{\tau},\tau_{\Omega}) &\cong
(A,\tau)*(\cT,\omega),\\
(\cS_{\tau},\tau_{\Omega}) &\cong
(A,\tau)*(C([-2,2]),\sigma),
\end{align*}
where $(\cT,\omega)$ denotes the classical Toeplitz C*-algebra with the vacuum state and $\sigma$ denotes  the state on $C([-2,2])$   with a semicircular distribution of variance 1.

\subsection{Relative cone lifting}
Given unital C*-algebras $A\subseteq C$, we define
$$
\cD(A,C)
=\{f\in C([0,1],C):f(0)\in A\}.
$$
We denote by $\ev_1\colon\cD(A,C)\to C$ the
evaluation at $1$ map.

Let $Q$ be a unital C*-algebra. An ideal $J\lhd Q$ is called a $\sigma$-ideal if for every separable C*-subalgebra $B\subseteq Q$ there exists a positive contraction $d\in B'\cap J$ such that $db=b$  for all $b\in B\cap J$.

The proof of the following proposition closely follows the proof of \cite[Proposition~1.17]{kirchberg-abel}.

\begin{proposition}\label{prop:relative-cone-lift}
Let $Q$ be a unital C*-algebra, let $J\lhd Q$ be a $\sigma$-ideal, and let
$p\colon Q\to Q/J$ be the quotient map. Let $A\subseteq Q$ and $C\subseteq Q/J$ be
separable and unital C*-subalgebras such that $A\cap J=\{0\}$  and 
$p(A)\subseteq C$.
Identifying $A$ with $p(A)\subseteq C$, there exists a unital
$*$-homomorphism $\Theta\colon\cD(A,C)\to Q$
fixing $A$ and such that 
$p\circ\Theta=\ev_1$.
Equivalently, the following diagram commutes:
$$\label{liftdiagram}
\begin{tikzcd}
A \arrow[r,hook] \arrow[d,equal]
& \cD(A,C) \arrow[r,"\ev_1"] \arrow[d,"\Theta"]
& C \arrow[d,hook] \\
A \arrow[r,hook]
& Q \arrow[r,"p"]
& Q/J .
\end{tikzcd}
$$
\end{proposition}

\begin{proof}
Choose a unital separable C*-subalgebra $Q_0\subseteq Q$ containing $A$ and such that $p(Q_0)=C$. Set $I=Q_0\cap J$.
Since $J$ is a $\sigma$-ideal, there exists a positive contraction
$d\in Q_0'\cap J$ such that $db=b$ for all $b\in I$.
Set $h=1-d$. Then $h$ is a positive contraction such that $h\in Q_0'$, $h\perp I$, and $p(h)=1$.

Since $C^*(1,h)$ and $Q_0$ commute, there is a unital $*$-homomorphism
$$ 
\alpha\colon C([0,1])\otimes Q_0\to Q,\qquad
\alpha(f\otimes c)=f(h)c.
$$
Here we use the nuclearity of $C([0,1])$.

Pointwise application of $p|_{Q_0}\colon Q_0\to C$ induces a $*$-homomorphism
$$
\widehat p\colon\cD(A,Q_0)\to\cD(A,C).
$$
Writing $\cD(A,C)=A+C_0((0,1])\otimes C$, we readily see
that $\widehat p$ is surjective. Moreover, if
$\widehat p(f)=0$, then $f$ is $I$-valued and $f(0)=0$, since
$A\cap I=\{0\}$. Thus
$$
\ker\widehat p=C_0((0,1])\otimes I.
$$

On the other hand, if $f\in C_0((0,1])$ and $b\in I$, then
$hb=0$ and hence $f(h)b=0$. It follows that
$\alpha$ vanishes on $C_0((0,1])\otimes I$. Hence, the
restriction of $\alpha$ to $\cD(A,Q_0)$ factors through
$\widehat p$, giving a unital $*$-homomorphism
$\Theta\colon \cD(A,C)\to Q$
 such that 
 $$
 \alpha|_{\cD(A,Q_0)}=\Theta\circ\widehat p.
 $$

For $f\in\cD(A,Q_0)$, since $p(h)=1$, we have $p(\alpha(f))=p(f(1))$.
It follows that $p\circ\Theta=\ev_1$.
Finally, for $a\in A$, viewed as a constant function, $\Theta(a)=a$, so $\Theta$ fixes $A$.
\end{proof}

\section{Part I}
The proof of Theorem \ref{mainthm} can be reduced to the case that $A$ is separable by considering separable elementary submodels of a given $A$. We thus fix a separable, infinite-dimensional, simple, unital C*-algebra $A$ with unique trace $\tau$ and strict comparison with respect to $\tau$. We choose a free ultrafilter  $\cU$ over $\N$. 

Let $M=\pi_{\tau}(A)''$, which is a $\mathrm{II}_1$ factor with separable predual. Let $M^{\cU}$ denote the tracial ultrapower of $M$.

Recall that we write $\cS_{\tau}=C^*(A,s_{\tau})$,
and that we denote by $\tau_{\Omega}$ the tracial state on $\cS_{\tau}$ induced by the vacuum state. 
By the reduced-free-product description above and \cite[Theorem~2]{dykema}, $\cS_{\tau}$ is a simple C*-algebra
with unique trace $\tau_{\Omega}$.

The following lemma is an immediate consequence of Popa's theorem \cite{Popa}.

\begin{lemma}
There is a unital trace-preserving embedding $\iota\colon \cS_{\tau}\to M^{\cU}$ fixing $A$.	
\end{lemma}	 

\begin{proof}
Choose a semicircular element $s\in M$ and, using Popa's theorem \cite{Popa}, a Haar unitary
$u\in M^{\cU}$ that is freely independent from $M$. Set $s_1=u^{-1}su$. Then $A\subseteq M$ and $C^*(1,s_1)$ are freely independent in $M^{\cU}$. Since the trace $\tau_{M^{\cU}}$ is faithful on $M^{\cU}$, 
$$
C^*(A,s_1)\cong A*C^*(1,s) 
\cong \cS_{\tau},
$$
where the free products are reduced.
\end{proof}		

Let $\cD=\cD(A,\cS_{\tau})$, and let $t\in \cD$ denote the identity on $[0,1]$. 
We regard $\cD$ endowed with the trace $\tau_\Omega\circ \ev_1$.

By \cite[Theorem 3.3]{KR}, we have an onto map $p\colon A^{\cU}\to M^{\cU}$ whose kernel is the trace-kernel ideal $J_\tau\lhd A^{\cU}$ formed by all $x\in A^{\cU}$ such that $\tau^{\cU}(x^*x)=0$.
By \cite[Proposition~4.6 and Remark~4.7]{KR}, $J_\tau$ is a $\sigma$-ideal.

\begin{lemma}
There is a unital $*$-homomorphism $\Theta\colon \cD\to A^{\cU}$ fixing $A$ and such that $p\circ\Theta=\iota \circ \ev_1$ and $\tau^{\cU}\circ\Theta=\tau_\Omega\circ \ev_1$. 
In other words, the following diagram is commutative and all the $*$-homomorphisms are trace preserving: 
\begin{equation}\label{diagram}
\begin{tikzcd}
A \arrow[r,hook] \arrow[d,equal]
& \cD \arrow[r,"\ev_1"] \arrow[d,"\Theta"]
& \cS_{\tau} \arrow[d,hook,"\iota"] \\
A \arrow[r,hook]
& A^{\cU} \arrow[r,"p"]
& M^{\cU}.
\end{tikzcd}
\end{equation}
\end{lemma}

\begin{proof}
Identify $\cS_{\tau}$ with $\iota(\cS_{\tau})\subseteq M^{\cU}=A^{\cU}/J_\tau$ and apply Proposition \ref{prop:relative-cone-lift}.
The equality $\tau^{\cU}\circ\Theta=\tau_\Omega\circ \ev_1$ follows by applying $\tau_{M^{\cU}}$ in $p\circ\Theta=\iota \circ \ev_1$ and using that $\iota$ is trace preserving.
\end{proof}

\section{Part II}
Set $B=\Theta(\cD)$ and $e=\Theta(1-t^2)\in B_+\cap J_\tau$. Note that, since
$\cD$ is separable, so is $B$.

\begin{lemma}
For every finite set $\Sigma\subset B$ and $\epsilon>0$ there exist positive contractions $c,c'\in B$  such that $cc'=c'$, $\tau^{\cU}(c')>0$, and 
$$
\|cbc - \tau^{\cU}(b)c^2\|<\epsilon,\qquad b\in \Sigma.
$$	
\end{lemma}	

\begin{proof}
Let $\Sigma=\{b_1,\ldots,b_m\}$. Choose lifts $f_k\in \cD$ of each $b_k$. Since $\cS_{\tau}$ is a simple C*-algebra, its trace $\tau_\Omega$ can be excised by positive contractions \cite[Proposition~2.3]{AAP}. Thus, there exists a positive $d\in \cS_{\tau}$ of norm 1 such that 
$$
\|df_k(1)d-\tau_\Omega(f_k(1))d^2\|<\epsilon/8,\qquad k=1,\ldots,m.
$$
We modify $d$ via functional calculus: choose $\chi\in C_0((0,1])$ that is zero on $[0,1/2]$ and $1$ on $[3/4,1]$, and linear on $[1/2,3/4]$, and then set
$$
d_1=\chi(d),\qquad d_1'=(d-3/4)_+.
$$
Then $d_1d_1'=d_1'$ and $d_1'\neq 0$. Writing $\chi(t)=t\xi(t)$ with $\|\xi\|\leq 2$, the preceding estimate gives
$$
\|d_1f_k(1)d_1-\tau_\Omega(f_k(1))d_1^2\|<\epsilon/2.
$$
Now pick $0\leq t_0<1$ such that 
$\|f_k(1)-f_k(t)\|<\epsilon/2$
for all $t\in [t_0,1]$.
Pick $g,g'\in C_0((0,1])$ supported in $[t_0,1]$, such that $0\leq g,g'\leq 1$,  $gg'=g'$ and $g(1)=g'(1)=1$. Set 
$$
c=\Theta(g(t)d_1),\qquad c'=\Theta(g'(t)d_1').
$$
Then $cc'=c'$, $c$ satisfies the inequality of the lemma,  and $\tau^{\cU}(c')=\tau_\Omega(d_1')>0$,
since $d_1'\neq 0$ and $\tau_\Omega$ is faithful
on $\cS_{\tau}$.
\end{proof}	

\begin{lemma}
There exists $r\in A^{\cU}$ such that
$$
r^*br=\tau^{\cU}(b)e,\qquad b\in B.
$$
\end{lemma}

\begin{proof}
Let $\Sigma\subset B$ be a finite set with $1\in\Sigma$ and let $\epsilon>0$.
Let $\eta>0$. Let $c,c'\in B$ be supplied by the previous lemma such that 
$$
\|cbc - \tau^{\cU}(b)c^2\|<\eta,\qquad b\in \Sigma.
$$
Since $A$ has strict comparison by $\tau$, its ultrapower $A^{\cU}$ has strict comparison by $\tau^{\cU}$ (\cite[Lemma 1.23]{bosaetal}, \cite[Theorem 8.2.2]{farahetal}). From
$$
d_{\tau^{\cU}}(e)=0<\tau^{\cU}(c')\leq d_{\tau^{\cU}}(c'),
$$
we deduce that $e\precsim_{\mathrm{Cu}}c'$. Thus, for every $\delta>0$ there exists
$v\in A^{\cU}$ such that
$$
\|e-v^*v\|<\delta,\qquad v\in\overline{c'A^{\cU}}.
$$
Since $cc'=c'$, we have $cv=v$. Moreover,
$\|v\|^2<1+\delta$.

For $b\in\Sigma$,
$$
v^*bv-\tau^{\cU}(b)e
=
v^*(cbc-\tau^{\cU}(b)c^2)v
+(v^*v-e)\tau^{\cU}(b).
$$
Thus,
$$
\|v^*bv-\tau^{\cU}(b)e\|
\leq
\eta\|v\|^2
+\delta|\tau^{\cU}(b)|.
$$
Choosing $\eta$ and $\delta$ sufficiently small, the right-hand side can be made arbitrarily small, uniformly for $b\in\Sigma$.

The conclusion now follows from Kirchberg's $\epsilon$-test \cite{kirchberg-abel}.
\end{proof}

\section{Part III}

Recall that in Part I we have obtained a $*$-homomorphism $\Theta\colon \cD\to A^{\cU}$ fixing $A$ and such that 
$$
\tau^{\cU}\circ \Theta=\tau_\Omega\circ\ev_1.
$$
Recall that in Part II we have obtained $r\in A^{\cU}$ such that
\begin{equation}\label{rstarThetar}
r^*\Theta(f)r=\tau^{\cU}(\Theta(f))e,\qquad f\in \cD,
\end{equation}
where $e=\Theta(1-t^2)$.

Let us define a cp map  $\phi\colon \cD\to \cD$ by
$$
\phi(f)=\tau_\Omega(f(1))(1-t^2).
$$
Let $\cX$ denote the C*-correspondence over $\cD$ obtained from the KSGNS construction applied  to $\phi$ (see \cite{brown-ozawa}). Let $\xi\in \cX$ be its canonical generator (i.e., the class of $1\otimes 1$).

We note that
$$
\cX\otimes_{\cD}\cX=0.
$$
Indeed, this readily follows from the fact that inner products in $\cX$ take values in the two-sided ideal $(1-t^2)\cD$, which is contained in the kernel of $\phi$. Consequently, the Fock space of $\cX$ is $\cF(\cX)=\cD\oplus \cX$. Let $\cT(\cX)$ be the Toeplitz-Pimsner
C*-algebra associated to $\cX$ and $\ell_{\xi}\in \cT(\cX)$ the left creation operator associated to $\xi$.

\begin{lemma}
There exists $\Psi\colon \cT(\cX)\to A^{\cU}$ extending $\Theta$ and such that $\Psi(\ell_{\xi})=r$.
\end{lemma}

\begin{proof}
This follows from the universal property of $\cT(\cX)$. Indeed, we can rewrite \eqref{rstarThetar} as
$$
r^*\Theta(f)r=\Theta(\phi(f)).
$$
Thus, mapping $\xi\mapsto r$ extends to a representation of $\cX$ in $A^{\cU}$ with underlying $*$-homomorphism  $\Theta$. By the universal property of $\cT(\cX)$, we get $\Psi$.
\end{proof}

We will show next that there exists a unital $*$-homomorphism 
$\cS_{\tau}\to\cT(\cX)$ fixing $A$. Composing it with $\Psi$
yields the desired $*$-homomorphism of $\cS_{\tau}$ into $A^{\cU}$ witnessing selflessness of $A$.

Recall that $\cT_{\tau}$ denotes the Toeplitz-Pimsner algebra associated to $\tau$. Define 
$$
\widetilde \cD=\{f\in C([0,1],\cT_{\tau}):f(0)\in A\}.
$$
Since $\cS_{\tau}\subseteq \cT_{\tau}$,  we have $\cD\subseteq \widetilde \cD$. Let $\widetilde \phi\colon \widetilde \cD\to \widetilde \cD$ be the cp map defined by
$$
\widetilde\phi(f)=\tau_{\Omega}(f(1))(1-t^2),\qquad f\in \widetilde\cD.
$$
Here $\tau_{\Omega}$ denotes the vacuum state on $\cT_{\tau}$, which extends the trace $\tau_{\Omega}$ on $\cS_{\tau}$.
It is clear that $\widetilde \phi$ extends $\phi$. Let $\widetilde \cX$ denote the C*-correspondence over $\widetilde \cD$ 
obtained from the KSGNS construction applied to the cp map $\widetilde\phi$. Let $\widetilde\xi\in\widetilde\cX$ be its canonical generator. As before, we have
 $$
\widetilde\cX\otimes_{\widetilde \cD}\widetilde\cX = 0.
 $$
Thus, $\cF(\widetilde\cX)=\widetilde\cD\oplus \widetilde \cX$.

\begin{lemma}
The assignment $f\xi g\mapsto f\widetilde \xi g$, for $f,g\in \cD$, extends to a covariant embedding $\cX\to \widetilde\cX$ with underlying
$*$-homomorphism the inclusion $\cD\to\widetilde\cD$.
The assignment $f\mapsto f$, $\ell_{\xi}\mapsto \ell_{\widetilde\xi}$ extends to an embedding of $\cT(\cX)$
into $\cT(\widetilde\cX)$.
\end{lemma}

\begin{proof}
Since $\cX$ is generated as a closed $\cD$-bimodule by $\xi$, by \cite[Lemma~2.1]{gaojungeetal}
it suffices to check that the assignment
$$
\xi\longmapsto \widetilde\xi
$$
preserves the corresponding coefficient map. For $f\in\cD$,
$$
\langle \widetilde\xi,f\widetilde\xi\rangle
=\widetilde\phi(f)
=\phi(f)
=\langle \xi,f\xi\rangle.
$$
Thus $\xi\mapsto\widetilde\xi$ extends to a covariant embedding
$\cX\to\widetilde\cX$ with underlying $*$-homomorphism
$\cD\hookrightarrow\widetilde\cD$.

The universal property of $\cT(\cX)$ now gives a $*$-homomorphism
$$
\cT(\cX)\longrightarrow\cT(\widetilde\cX),
\qquad
f\mapsto f,\quad
\ell_\xi\mapsto\ell_{\widetilde\xi}.
$$
To see that it is injective, note that $\cD\oplus\cX$, viewed as a submodule of
$\widetilde\cD\oplus\widetilde\cX$, is reducing for the left action of $\cD$
and for $\ell_{\widetilde\xi}$. Indeed,
$\ell_{\widetilde\xi}$ maps $\cD$ into $\cX$, while it vanishes on
$\cX$ because
$\widetilde\cX\otimes_{\widetilde\cD}\widetilde\cX=0$; similarly,
$\ell_{\widetilde\xi}^*$ maps $\cX$ into $\cD$ and vanishes on
$\cD$. Moreover, under the identification
$\cD\oplus\cX\subseteq\widetilde\cD\oplus\widetilde\cX$, the
restrictions of the left action of $\cD$ and of
$\ell_{\widetilde\xi}$ are precisely the left action of $\cD$ and
$\ell_\xi$ on $\cF(\cX)$, i.e., the Fock
representation of $\cT(\cX)$. Hence the $*$-homomorphism is injective.
\end{proof}

Identify $\cT(\cX)$ as a C*-subalgebra of $\cT(\widetilde\cX)$ by the previous embedding and regard $\ell_{\widetilde \xi}$ as an extension
of $\ell_{\xi}$.

\begin{lemma}
Let $W\in \cT(\widetilde\cX)$ be defined as
$$
W=t\ell_{\tau} + \ell_{\widetilde \xi}.
$$
\begin{enumerate}[(i)]
\item
We have 
$$
W^*aW=\tau(a)1,\qquad a\in A,
$$
and $W+W^*\in \cT(\cX)$.
\item
The assignment $A\ni a\mapsto a$ and $\ell_{\tau}\mapsto W$ extends to a $*$-homomorphism from $\cT_{\tau}$
into $\cT(\widetilde\cX)$.
\item
The assignment $A\ni a\mapsto a$ and $s_{\tau}\mapsto W+W^*$ extends to a $*$-homomorphism from $\cS_{\tau}$
into $\cT(\cX)$.

\end{enumerate}
\end{lemma}

\begin{proof}
(i) We have
\begin{align*}
W^*aW &=(t\ell_{\tau})^*a(t\ell_{\tau}) +
(t\ell_{\tau})^*a\ell_{\widetilde \xi}
+ \ell_{\widetilde\xi}^*a(t\ell_{\tau}) + \ell_{\widetilde\xi}^*a\ell_{\widetilde\xi}\\
&=t^2\tau(a) + 0 + 0 + (1-t^2)\tau(a)\\
&=\tau(a)1.
\end{align*}
To see that the cross terms are zero, we used that
$$
(t\ell_{\tau})^*a\ell_{\widetilde \xi}=
\ell_{(t\ell_{\tau})^*a\widetilde \xi}
$$
and that $(t\ell_{\tau})^*a\widetilde \xi=0$, which follows from
$$
\langle (t\ell_{\tau})^*a\widetilde \xi,(t\ell_{\tau})^*a\widetilde \xi\rangle 
=\widetilde\phi(t^2a^*\ell_{\tau}\ell_{\tau}^*a)=
\tau_{\Omega}(a^*\ell_{\tau}\ell_{\tau}^*a)(1-t^2)=0.
$$

Finally, to see that $W+W^*\in \cT(\cX)$, notice that $\ell_{\widetilde\xi}\in \cT(\cX)$, since we have identified $\ell_{\xi}\in \cT(\cX)$ with $\ell_{\widetilde\xi}$, and that
\begin{align*}
t\ell_{\tau}+t\ell_{\tau}^*=ts_{\tau}\in \cD\subseteq\cT(\cX). 
\end{align*}

(ii) This follows from (i) and the universal property of $\cT_{\tau}$.

(iii) Since $W+W^*\in \cT(\cX)$, the restriction of the $*$-homomorphism from (ii) to $\cS_{\tau}$ ranges in $\cT(\cX)$.
\end{proof}

\begin{proof}[Proof of Theorem \ref{mainthm}]
Let $(A,\tau)$ be as in the theorem and assume additionally that $A$ is separable.
By the previous lemma, there exists a unital
$*$-homomorphism $\cS_{\tau}\to \cT(\cX)$ fixing $A$, which sends $s_\tau$ to
$W+W^*=ts_\tau+\ell_\xi+\ell_\xi^*$. Combined with $\Psi\colon \cT(\cX)\to A^{\cU}$,
we obtain $j\colon \cS_{\tau}\to A^{\cU}$ fixing $A$ and sending $s_\tau$ to
$\Theta(ts_\tau)+r+r^*$. Preservation of the respective traces is automatic since $\cS_{\tau}$ has a unique trace (by \cite[Theorem~2]{dykema}). It follows that $A$ is selfless by the reduced-free-product description of $\cS_{\tau}$ and \cite[Theorem~2.6(iii)]{selfless}. 

Let us now drop the assumption that $A$ is separable.  The properties of simplicity, infinite dimensionality,  uniqueness of trace, and strict comparison by the trace all
pass to any elementary submodel $(A',\tau|_{A'})\subseteq (A,\tau)$ (see \cite[Proposition 5.10.3, Theorem 8.2.2]{farahetal}). Thus, every separable elementary submodel of $A$ is selfless.
Passing to the direct limit over all its separable elementary submodels, we deduce that $(A,\tau)$ is also selfless \cite[Theorem 4.1]{selfless}. 
\end{proof}	

\begin{proof}[Proof of Corollary \ref{coro:HTO}]
It is well known that the Jiang--Su C*-algebra
is simple, has a unique trace, and has strict comparison
with respect to its trace \cite{jiang-su}.
By Theorem \ref{mainthm}, it is selfless.
By \cite[Theorem~2.6]{selfless}, there is an embedding
of $\mathcal Z*C_r^*(\mathbb F_\infty)$ in $\mathcal Z^{\cU}$.
In particular, $C_r^*(\mathbb F_2)$ embeds in $\mathcal Z^{\cU}$. Since $\mathcal Z$ embeds in $\mathcal Q$ (as the latter is unital and $\mathcal Z$-stable), $C_r^*(\mathbb F_2)$ embeds in $\mathcal Q^{\cU}$.

\end{proof}

\end{document}